\documentclass{amsart}

\usepackage{graphics}
\usepackage{amsmath,amsfonts,latexsym,amssymb}
\usepackage{comment}
\usepackage{xcolor}
\usepackage{fullpage}

\usepackage[english]{babel}
\usepackage{amsthm}
\usepackage{mathrsfs}
\usepackage{pgf,tikz}
\usepackage{ulem}
\usepackage{amstext} 
\usepackage{array}   
\newcolumntype{C}{>{$}c<{$}} 
\definecolor{uququq}{rgb}{0.25,0.25,0.25}
\newtheorem{thm}{Theorem}[section]
\newtheorem{cor}[thm]{Corollary}

\newtheorem{prop}[thm]{Proposition}

\usetikzlibrary{arrows}
\theoremstyle{definition}
\newtheorem{defn}[thm]{Definition}
\theoremstyle{definition}

\theoremstyle{definition}
\newtheorem{ex}[thm]{Example}

\newcommand{\Z}{\mathbb{Z}}

\newcommand{\F}{\mathbb{F}}
\def\H{\mathrm{H}}

\begin{document}

\title{More Heffter Spaces via finite fields}

\author{Marco Buratti}
\address{SBAI - Sapienza Universit\`a di Roma, Via Antonio Scarpa 16, I-00161 Roma, Italy}
\email{marco.buratti@uniroma1.it}

\author{Anita Pasotti}
\address{DICATAM - Sez. Matematica, Universit\`a degli Studi di Brescia, Via
Branze 43, I-25123 Brescia, Italy}
\email{anita.pasotti@unibs.it}
%
%

\subjclass[2010]{}

\maketitle

\begin{abstract}
A $(v,k;r)$ Heffter space is a resolvable $(v_r,b_k)$ configuration whose points form a half-set of an abelian group $G$
and  whose blocks are all zero-sum in $G$.
It was recently proved that there are infinitely many orders $v$ for which, given any pair $(k,r)$ with $k\geq3$ odd,
a $(v,k;r)$ Heffter space exists. This was obtained by imposing a point-regular automorphism group. Here we relax this request by asking
for a point-semiregular automorphism group. In this way the above result is extended also to the case $k$ even.
\end{abstract}

\section{Introduction}

We first recall some notions of design theory. A point-block incidence structure
is a {\it partial linear space} if any two distinct points are contained in at most one block.
A $(v_r,b_k)$ {\it configuration} is a partial linear space $(V,{\mathcal B})$ with $v$ points and $b$ blocks each of size $k$
such that every point is contained in exactly $r$ blocks. It is resolvable if there exists a partition ({\it resolution}) of ${\mathcal B}$ 
into $r$ classes ({\it parallel classes}) each of which is in its turn a partition of $V$.
In a $(v_r,b_k)$ configuration we necessarily have $vr=bk$. 
If it is resolvable, then $v$ is obviously divisible by $k$.
An automorphism group of a resolvable configuration is a group of permutations on the point set leaving
the resolution invariant.
We refer to \cite{BS22,G} for some results of resolvable configurations.

A point-block incidence structure $(V,{\mathcal B})$ is said to be {\it $G$-additive} if $V$ is a subset of an abelian group $G$
and each block is zero-sum in $G$. This notion, firstly introduced in \cite{CFP} for classic $t$-$(v,k,\lambda)$ designs,
gave rise to an extensive literature \cite{BN1,BN2,CFP2,FP}, and more recently it was extended to other combinatorial designs \cite{BMN}.
Heffter spaces have been defined very recently in \cite{BP2} as a generalization of the well-known {\it Heffter arrays}.
\begin{defn}
A $(v,k;r)$ {\it Heffter space} over an abelian group $G$ is a $G$-additive and resolvable $(v_r,b_k)$ configuration with point set $V$ such
that $V$ and $-V$ partition $G\setminus\{0\}$ (briefly, $V$ is a {\it half-set} of $G$). 
It is {\it simple} if every block can be ordered in such a way that its partial sums are pairwise distinct.
\end{defn}

A more general definition of a {\it relative Heffter space} will be proposed in \cite{BJMP}.

Note that in a  $(v,k;r)$ Heffter space we necessarily have $v=kn$ for some integer $n$ since it is resolvable.

A $(kn,k;1)$ Heffter space is nothing but a $(kn,k)$ {\it Heffter system}, that is a partition of a half-set of an abelian 
group of order $2kn+1$ into zero-sum parts of size $k$. 
As explained in \cite{BP2},
the existence of a simple $(kn,k)$ Heffter system over $\Z_{2kn+1}$ for any admissible $k\geq3$
can be deduced from well-known results on {\it cyclic cycle systems} \cite{BGL,BurDel}.
Two Heffter systems with the same point set $V$ are said to be {\it orthogonal} if each block of a system
shares at most one element with each block of the other. 
Thus a $(kn,k;r)$ Heffter space is equivalent to a set of $r$ {\it mutually orthogonal} $(kn,k)$ Heffter systems.

A $(kn,k;2)$ Heffter space with parallel classes ${\mathcal P}=\{B_1,\dots,B_n\}$ and ${\mathcal P}'=\{B'_1,\dots,B'_n\}$,
is equivalent to the $n\times n$ matrix $A$ whose cell $a_{i,j}$ is empty if $B_i$ and $B'_j$ are disjoint or contains their common element otherwise.
Such a matrix $A$ is a {\it Heffter array} with parameters $n$ and $k$ and it is typically denoted by H$(n;k)$.

Heffter arrays have been introduced by Archdeacon \cite{A} for their important applications, especially in topological graph theory.
It was proved with the contribution of several authors \cite{ADDY,CDDY,DW}
that there exists an $\H(n; k)$ (or equivalently a $(kn,k;2)$ Heffter space) if and only if $n \geq k \geq  3$.
For a survey on Heffter arrays we refer to \cite{PD}.

It has been conjectured \cite{CMPPSums} that every zero-sum subset $B$ of an abelian group having no pair of opposite elements
and not containing zero admits an ordering whose related partial sums are pairwise distinct.
The conjecture has been proved for $|B|\leq10$ (see \cite{AL,CMPPSums}) so that every Heffter space with block size $k\leq10$
is automatically simple.

The {\it density} of a $(v,k;r)$ Heffter space is the  ratio $\delta={r(k-1)\over v-1}$ 
which measures how much the space is far from being a {\it Steiner $2$-design}, case in which $\delta=1$.
The densest Heffter space constructed in \cite{BP2} has parameters $(121,11;9)$ with density $\delta=0.75$. Recently,
a more dense Heffter space has been obtained \cite{N}; its parameters are $(105,5;21)$ with density $\delta\simeq0.81$.
The existence of a Heffter space with density 1 is still an open question.
For an extensive discussion on this problem, see Section 2 in \cite{BP2}.

\medskip
Finite fields have been used in \cite{B} for the construction of {\it rank-one} Heffter arrays and
in \cite{BP2} for the construction of Heffter spaces with an automorphism group acting regularly on the points.
In this paper finite fields are exploited again to construct $({q-1\over2},k;r)$ Hefftter spaces with $q$ a prime power. 
This time our search is done with the weaker request of an automorphism group of order
${q-1\over2k}$ acting semiregularly on the points.
In this way we get the following main result which in \cite{BP2} was obtained only for $k$ odd.

\begin{thm}\label{main}
For any $k\geq 3$ and any $r\geq0$, there are infinitely many values of $v$ for which there exists a $(v,k;r)$ Heffter space.
\end{thm}

We note that, for the time being, for all known $(v,k;r)$ Heffter spaces with $r>2$ we have that $2v+1$ is a prime power.
In a forthcoming paper \cite{BP3}, with completely different methods, we will present some new constructions most of which do not have this constraint.

\section{Heffter spaces via Heffter difference matrices}
As it is standard, the field of order $q$ and its multiplicative group will be denoted by $\F_q$ and $\F_q^*$, respectively.
We also agree to use the following notation. 

Given a prime power $q\equiv2k+1$ (mod $4k$), it will be tacitly assumed that a primitive element $g$ of $\F_q$ is fixed. 
The subgroup of $\F_q^*$ of index $2k$ will be denoted by $C^{2k}$ and its cosets in $\F_q^*$,
which are $g^0C^{2k}$, $g^1C^{2k}$, \dots, $g^{2k-1}C^{2k}$, are called the {\it cyclotomic classes of $\F_q$ of order $2k$}. 
They will be denoted by $C^{2k}_0$, $C^{2k}_1$, \dots, $C^{2k}_{2k-1}$, respectively.
Note that the hypothesis $q\equiv2k+1$ (mod $4k$) implies that ${q-1\over2k}$ is odd so that $-1\in C^{2k}_k$.
It follows that any set of the form 
\begin{equation}\label{cyclotomic half-set}
(-1)^{\alpha_0} C^{2k}_0 \ \cup \ (-1)^{\alpha_1} C^{2k}_1 \ \cup \ \dots \ \cup \ (-1)^{\alpha_{k-1}} C^{2k}_{k-1}
\end{equation}
with the $\alpha_i$'s in $\{0,1\}$ is a half-set of $(\F_q,+)$. We will call it a {\it cyclotomic half-set of $\F_q$ of order $2k$}.

In particular, if $q\equiv3$ (mod 4), the cyclotomic half-sets of $\F_q$ of order $2$ are simply the sets $C^2_0$ and $C^2_1$
of the non-zero squares and non-squares of $\F_q$. 
In general, for $q\equiv 2k+1$ (mod $4k$), there are precisely $2^k$ cyclotomic half-sets of $\F_q$ of order $2k$.

In the following, given a cyclotomic half-set $V$ as in (\ref{cyclotomic half-set}),
we set $V_i=(-1)^{\alpha_i} C^{2k}_i$, for $0\leq i\leq k-1$.
It will be also assumed that $V_0=C^{2k}_0$ rather than $-C^{2k}_0$.

The crucial tool for getting our main result is
an $r\times k$ matrix $B$ with entries in a cyclotomic half-set of order $2k$. 
We agree to denote the $h$-th row and the $i$-th column of such a matrix $B$ by $B_h$ and $B^i$, respectively.
Also, for any pair of columns $B^i$ and $B^j$ we will consider the multiset of ratios
$${B^i/B^j}:=\biggl{\{}{b_{h,i}\over b_{h,j}} \ | \ 1\leq h\leq r\biggl{\}}.$$

\begin{defn}\label{HM}
Let $q\equiv 2k+1$ $($mod $4k)$ be a prime power and let $V=V_0 \ \cup \ \dots \ \cup \ V_{k-1}$ be a cyclotomic half-set of $\F_q$.
A {\it $(V,k,r)$ Heffter difference matrix} is an $r\times k$ matrix $B$ (rows indexed from 1 to $r$, columns indexed from 0 to $k-1$)
with the following properties:
\begin{itemize}
\item[$(i)$] every row $B_h$ is zero-sum in $\F_q$;
\medskip
\item[$(ii)$] every element of the column $B^i$ belongs to $V_{i}$ for $0\leq i\leq k-1$;
\medskip
\item[$(iii)$] ${B^i/B^j}$ does not have repeated elements for $0\leq i<j\leq k-1$.
\end{itemize}
The matrix $B$ is {\it simple} if the partial sums of every row $B_h$ are pairwise distinct.
\end{defn}

Who has familiarity with {\it difference matrices} (see, e.g., \cite{BJL} or \cite{CD}) will recognize that the transpose of a $(V,k,r)$ Heffter difference matrix
is, so to say, a ``partial" $(q-1,k,1)$ difference matrix in $\F_q^*$.

We say that a $(V,k,r)$ Heffter difference matrix is {\it normalized} if $B^0$ is the all-one column.
Note that if $B$ is a $(V,k,r)$ Heffter difference matrix, then the matrix $B'$ obtained from $B$ 
by dividing the $h$-th row by $b_{h,0}$, for $1\leq h\leq r$, is a normalized $(V,k,r)$ Heffter difference matrix.
Clearly, $B'$ is simple if and only if $B$ is simple as well.
Thus the search for (possibly simple) Heffter difference matrices can be done up to normalizations.

Let us show how Heffter difference matrices give rise to Heffter spaces.

\begin{thm}\label{general}
Let $q\equiv 2k+1$ $($mod $4k)$ be a prime power and let $V=V_0 \ \cup \ \dots \ \cup \ V_{k-1}$ be a cyclotomic half-set of $\F_q$.
If there exists a $(V,k,r)$ Heffter difference matrix,
then there exists a $({q-1\over2},k;r)$ Heffter space. If the matrix is simple, the space is simple as well.
\end{thm}
\begin{proof}
Let $B$ be a $(V,k;r)$ Heffter difference matrix.
For $1\leq h\leq r$, let ${\mathcal P}_h=\{tB_h \ | \ t\in C^{2k}\}$ be the orbit of $B_h$ under the natural multiplicative action of $C^{2k}$.
Set $\displaystyle{\mathcal B}=\bigcup_{h=1}^r{\mathcal P}_h$ and consider the point-block incidence structure $(V,{\mathcal B})$.
To prove the first part of the assertion it is enough to show that this structure is an additive and resolvable partial linear space.
The second part will be obvious since the partial sums of each block $tB_h$ are nothing but the partial sums of $B_h$ multiplied by $t$.

\medskip
We first prove that $(V,{\mathcal B})$ is a partial linear space.

\medskip
Take any block $tB_h=\{tb_{h,0},tb_{h,1},\dots,tb_{h,k-1}\}\in{\mathcal B}$.
By property $(ii)$ of Definition \ref{HM} we have $b_{h,i}\in V_i$ for $0\leq i\leq k-1$. Then, observing that $C^{2k}$ fixes $V_i$ for every $i$, 
we have $tb_{h,i}\in V_i$ for $0\leq i\leq k-1$. Thus every block has exactly one element in $V_i$ for $0\leq i\leq k-1$.

Let $x$, $y$ be two distinct elements of $V$ so that we have $x\in V_i$ and $y\in V_j$ for a suitable pair $(i,j)$.
Assume that there are two blocks of $\mathcal B$, say $tB_h$ and $uB_\ell$,
containing $\{x,y\}$. 
We have to prove that they coincide, i.e., that we have $t=u$ and $h=\ell$.

First of all, we necessarily have $i\neq j$ otherwise the blocks $tB_h$ and $uB_\ell$
would contain two elements of $V_i$ contradicting what we established before.

Note that $x\in V_i$ and $x\in tB_h \ \cap \ uB_\ell$ implies that $x=tb_{h,i}=ub_{\ell,i}$.
Also, $y\in V_j$ and $y\in tB_h \ \cap \ uB_\ell$ implies that $y=tb_{h,j}=ub_{\ell,j}$.
These equalities give $\displaystyle{x\over y}={b_{h,i}\over b_{h,j}}={b_{\ell,i}\over b_{\ell,j}}$.
It follows that $h=\ell$ otherwise $B^i/B^j$ would have repeated elements contradicting property $(iii)$
of Definition \ref{HM}.
Now, from $h=\ell$ and the equality $tb_{h,i}=ub_{\ell,i}$ we also get $t=u$.

\medskip
It is easy to see that $(V,{\mathcal B})$ is additive.

\medskip
Indeed every block of ${\mathcal B}$ is a multiple of a row of $B$ which, by property $(i)$ of Definition \ref{HM},
is zero-sum in $\F_q$. It follows that every block of ${\mathcal B}$ is zero-sum as well.

\medskip

\medskip
We finally show that $(V,{\mathcal B})$ is resolvable.

\medskip
By property $(ii)$ of Definition \ref{HM}, every row $B_h$ of $B$ has precisely one element in $V_i$ for $0\leq i\leq k-1$.
Thus we have $\displaystyle\bigcup_{t\in C^{2k}}tB_h=V$
which means that the blocks of ${\mathcal P}_h$ partition $V$.
It follows that
${\mathcal R}:=\{{\mathcal P}_h \ | \ 1\leq h\leq r\}$
is a resolution of $(V,{\mathcal B})$ and the assertion follows.
\end{proof}

We note that the $({q-1\over2},k;r)$ Heffter space constructed in the above theorem has an automorphism group
isomorphic to $C^{2k}$ acting semiregularly on the points and fixing each parallel class. This is the group generated by the
permutation $\alpha: x\in V \longrightarrow g^{2k}x\in V$ 
where $g$ is a primitive element of $\F_q^*$.

We will prove our main result applying Theorem \ref{general} with $V$ the ``standard" cyclotomic half-set of order $2k$,
that is $C^{2k}_0 \ \cup \ C^{2k}_1 \ \cup \ \dots \ \cup \ C^{2k}_{k-1}$.
Anyway, for a fixed pair $(q,k)$, the use of other cyclotomic half-sets may help to obtain greater values of $r$ for which a $({q-1\over2},k;r)$ Heffter space exists.

In the following examples we determine the maximum $r$ for which Theorem \ref{general} allows to determine a $({q-1\over2},k;r)$ Heffter system
for three special instances of the pair $(q,k)$. This does not exclude that one may get a greater value of $r$ with another method, possibly by computer.

\begin{ex}\label{73}
Let us determine the maximum $r$ for which Theorem \ref{general} allows to get a $(36,4;r)$ Heffter space.
Note that $q=2\cdot36+1=73$ is a prime
and that $q\equiv 9$ (mod 16), i.e., $q\equiv 2k+1$ (mod $4k$) with $k=4$. 
We take $g=5$ as primitive element of $\F_{73}$.
The use of Theorem \ref{general} with $V$ the standard cyclotomic half-set of order 8 fails
to find a $(36,4;r)$ Heffter space with $r>2$. Indeed an exhaustive computer search has shown that
a $(V,4,3)$ Heffter difference matrix does not exist.
On the other hand now we show that the following is a $(W,4,4)$ Heffter difference matrix with $W=C^8_0 \ \cup \ -C^8_1 \ \cup \ C^8_2 \ \cup \ C^8_3$.
$$
B=\begin{array}{|r|r|r|r|} \hline
1 &  68 & 25 & 52  \\ \hline
1 &  53 & 49 & 43  \\ \hline
1 &  33 & 50 & 62  \\ \hline
1 &  59 & 35 & 51  \\ \hline
\end{array}$$
First of all one can check that each row is zero-sum in $\F_{73}$ so that property $(i)$ of Definition \ref{HM} holds.
Now consider the natural isomorphism $\phi$ between $\F_{73}^*$ and $\Z_{72}$ mapping
any element  $g^i$ into $i$. We have:
$$
\phi(B)=\begin{array}{|r|r|r|r|} \hline
0 &  37 & 2 & 3  \\ \hline
0 &  53 & 66 & 51  \\ \hline
0 &  61 & 10 & 19  \\ \hline
0 &  5 & 34 & 27  \\ \hline
\end{array}$$
We note that for $i=0,2,3$ the elements of $\phi(B^i)$ are all congruent to $i$ modulo 8 which
means that the elements of $B^i$ belong to $C^8_{i}$. Now note that the elements of $\phi(B^1)$ are all congruent to $5$ modulo 8
which means that the elements of $B^1$ belong to $C^8_{5}=-C^8_1$. Recalling the definition of $W$, we conclude that $B^i$ is contained in $W_i$ for $0\leq i\leq 3$,
i.e., property $(ii)$ of Definition \ref{HM} holds.

It is readily seen that the difference $\phi(B^0)-\phi(B^i)$ does not have repeated elements for $1\leq i\leq 3$. 
Note that $\phi(B^i)-\phi(B^j)$ does not have repeated elements also for $1\leq i<j\leq 3$:
$$\phi(B^1)-\phi(B^2)=(37,53,61,5)-(2,66,10,34)=(35, 59, 51, 43);$$
$$\phi(B^1)-\phi(B^3)=(37,53,61,5)-(3,51,19,27)=(34, 2, 42, 50);$$
$$\phi(B^2)-\phi(B^3)=(2,66,10,34)-(3,51,19,27)=(71, 15, 63, 7).$$
The above implies that $B^i/B^j$ does not have repeated elements for $0\leq i<j\leq 3$ that is property $(iii)$ of Definition \ref{HM}. 
Thus $B$ gives rise to a $(36,4;4)$ Heffter space.
Following the instructions given in the proof of Theorem \ref{general}, 
its resolution $\{{\mathcal P}_1,\dots,{\mathcal P}_4\}$ is the following.

\medskip
\begin{center}
\begin{tabular}{|l|c|r|c|r|c|r|c|r|c|r|c|r|}
\hline {\quad\quad\quad$\mathcal {P}_1$}    \\
\hline $\{1, 68, 25, 52\}$\\
\hline $\{2, 63, 50, 31\}$  \\
\hline $\{4, 53, 27, 62\}$\\
\hline $\{8, 33, 54, 51\}$  \\
\hline $\{16, 66, 35, 29\}$\\
\hline $\{32, 59, 70, 58\}$ \\
\hline $\{64, 45, 67, 43\}$ \\
\hline $\{55, 17, 61, 13\}$ \\
\hline $\{37, 34, 49, 26\}$ \\
\hline
\end{tabular}\quad\quad\quad
\begin{tabular}{|l|c|r|c|r|c|r|c|r|c|r|c|r|}
\hline {\quad\quad\quad$\mathcal{P}_2$}    \\
\hline $\{1, 53, 49, 43\}$\\
\hline $\{2, 33, 25, 13\}$  \\
\hline $\{4, 66, 50, 26\}$\\
\hline $\{8, 59, 27, 52\}$  \\
\hline $\{16, 45, 54, 31\}$\\
\hline $\{32, 17, 35, 62\}$ \\
\hline $\{64, 34, 70, 51\}$ \\
\hline $\{55, 68, 67, 29\}$ \\
\hline $\{37, 63, 61, 58\}$ \\
\hline
\end{tabular}\quad\quad\quad
\begin{tabular}{|l|c|r|c|r|c|r|c|r|c|r|c|r|}
\hline {\quad\quad\quad$\mathcal{P}_3$}    \\
\hline $\{1, 33, 50, 62\}$\\
\hline $\{2, 66, 27, 51\}$  \\
\hline $\{4, 59, 54, 29\}$\\
\hline $\{8, 45, 35, 58\}$  \\
\hline $\{16, 17, 70, 43\}$\\
\hline $\{32, 34, 67, 13\}$ \\
\hline $\{64, 68, 61, 26\}$ \\
\hline $\{55, 63, 49, 52\}$ \\
\hline $\{37, 53, 25, 31\}$ \\
\hline
\end{tabular}\quad\quad\quad
\begin{tabular}{|l|c|r|c|r|c|r|c|r|c|r|c|r|}
\hline {\quad\quad\quad$\mathcal{P}_4$}    \\
\hline $\{1, 59, 35, 51\}$\\
\hline $\{2, 45, 70, 29\}$  \\
\hline $\{4, 17, 67, 58\}$\\
\hline $\{8, 34, 61, 43\}$  \\
\hline $\{16, 68, 49, 13\}$\\
\hline $\{32, 63, 25, 26\}$ \\
\hline $\{64, 53, 50, 52\}$ \\
\hline $\{55, 33, 27, 31\}$ \\
\hline $\{37, 66, 54, 62\}$ \\
\hline
\end{tabular}\quad
\end{center}
\end{ex}

\medskip
\begin{ex}\label{71}
Let us determine the maximum $r$ for which Theorem \ref{general} allows to get a $(35,5;r)$ Heffter space.
Note that $q=2\cdot35+1=71$ is a prime
and that $q\equiv 11$ (mod 20), i.e., $q\equiv 2k+1$ (mod $4k$) with $k=5$. 
Taking $g=7$ as primitive element of $\F_{71}$, one can check that the following is a $(V,5,3)$ Heffter difference matrix
with $V$ the standard cyclotomic half-set of order 10 of $\F_{71}$:
$$
\begin{array}{|r|r|r|r|r|r|} \hline
1 &  7 & 38 & 42 & 54 \\ \hline
1 &  46 & 9 & 28 & 58 \\ \hline
1 &  52 & 6 & 59 & 24  \\ \hline
\end{array}$$
An exhaustive computer search has shown 
that there is no $(V,5,r)$ Heffter difference matrix with $r>3$.
On the other hand, using the cyclotomic half-set
$$W=C^{10}_0 \ \cup \ -C^{10}_1 \ \cup \ C^{10}_2 \ \cup \ -C^{10}_3 \ \cup \ C^{10}_4$$
it has been possible to find the following $(W,5,5)$ Heffter difference matrix
$$
\begin{array}{|r|r|r|r|r|r|} \hline
1 &  25 & 49 & 43 & 24 \\ \hline
1 &  60 & 38 & 27 & 16 \\ \hline
1 &  19 & 50 & 18 & 54  \\ \hline
1 &  3 & 4 & 5 & 58  \\ \hline
1 &  40 & 57 & 29 & 15  \\ \hline
\end{array}$$

\medskip\noindent
Then the above matrix gives rise to a $(35,5;5)$ Heffter space. One can check that it is the same Heffter space 
constructed in Example 4.5 in \cite{BP2}.
Note, in particular, that $-C^{10}_1=C^{10}_2$ and that $-C^{10}_3=C^{10}_8$ so that $W$ is nothing but $C^2$. 
\end{ex}

\medskip
\begin{ex}\label{109}
Let us determine the maximum $r$ for which Theorem \ref{general} allows to get a $(54,6;r)$ Heffter space.
Note that $q=2\cdot54+1=109$ is a prime
and that $q\equiv 13$ (mod 24), i.e., $q\equiv 2k+1$ (mod $4k$) with $k=6$. 
Taking $g=6$ as primitive element of $\F_{109}$, one can check that the following is a $(V,6,5)$ Heffter difference matrix
with $V$ the standard cyclotomic half-set of order 12 of $\F_{109}$:
$$
\begin{array}{|r|r|r|r|r|r|} \hline
1 &  6 & 36 & 19 & 5 & 42 \\ \hline
1 &  96 & 88 & 92 & 89 & 70 \\ \hline
1 &  52 & 31 & 33 & 3 & 98  \\ \hline
1 &  69 & 74 & 68 & 97 & 18  \\ \hline
1 &  14 & 60 & 86 & 7 & 50  \\ \hline
\end{array}$$
An exhaustive computer search has shown 
that there is no $(V,6,r)$ Heffter difference matrix with $r>5$.
On the other hand, a better result can be obtained using the cyclotomic half-set
$$W=C^{12}_0 \ \cup \ -C^{12}_1 \ \cup \ C^{12}_2 \ \cup \ -C^{12}_3 \ \cup \ C^{12}_4 \ \cup \ C^{12}_5.$$
Indeed one can check that the following is a $(W,6,6)$ Heffter difference matrix.
$$
\begin{array}{|r|r|r|r|r|r|} \hline
1 &  103 & 31 & 41 & 81 & 70 \\ \hline
1 &  13 & 74 & 54 & 26 & 50 \\ \hline
1 &  58 & 88 & 2 & 80 & 98  \\ \hline
1 &  56 & 36 & 76 & 5 & 44  \\ \hline
1 &  57 & 100 & 23 & 7 & 30  \\ \hline
1 &  95 & 60 & 17 & 3 & 42  \\ \hline
\end{array}$$

\medskip\noindent
Thus there exists a $(54,6;6)$ Heffter space. 
\end{ex}

\medskip
A computer search allowed us to find a $(V,6,6)$ Heffter difference matrix with $V$ the standard cyclotomic half-set of order 12 of $\F_q$
for all primes $q\equiv13$ (mod 24) greater than 109 and smaller than 8000 . We report our computer results for the first four values of $q$.

$$q=157:\quad 
\begin{array}{|r|r|r|r|r|r|} \hline
1 &  5 & 25 & 125 & 115 & 43 \\ \hline
1 &  22 & 13 & 129 & 89 & 60 \\ \hline
1 &  34 & 110 & 8 & 19 & 142  \\ \hline
1 &  137 & 105 & 116 & 17 & 95  \\ \hline
1 &  69 & 127 & 54 & 113 & 107  \\ \hline
1 &  61 & 51 & 65 & 81 & 55  \\ \hline
\end{array}\quad\quad\quad q=181:
\begin{array}{|r|r|r|r|r|r|} \hline
1 &  2 & 4 & 8 & 148 & 18 \\ \hline
1 &  47 & 37 & 155 & 44 & 78 \\ \hline
1 &  109 & 116 & 7 & 14 & 115  \\ \hline
1 &  118 & 147 & 35 & 38 & 23  \\ \hline
1 &  58 & 138 & 31 & 102 & 132  \\ \hline
1 &  96 & 94 & 74 & 121 & 157  \\ \hline
\end{array}
$$

\medskip
$$q=229:\quad 
\begin{array}{|r|r|r|r|r|r|} \hline
1 &  6 & 36 & 216 & 171 & 28 \\ \hline
1 &  74 & 209 & 143 & 183 & 77 \\ \hline
1 &  73 & 215 & 128 & 51 & 219  \\ \hline
1 &  137 & 154 & 109 & 129 & 157  \\ \hline
1 &  163 & 118 & 52 & 37 & 87  \\ \hline
1 &  102 & 174 & 145 & 83 & 182  \\ \hline
\end{array}\quad\quad\quad q=277:
\begin{array}{|r|r|r|r|r|r|} \hline
1 &  5 & 25 & 125 & 71 & 50 \\ \hline
1 &  135 & 220 & 262 & 136 & 77 \\ \hline
1 &  44 & 121 & 145 & 165 & 78  \\ \hline
1 &  80 & 196 & 51 & 202 & 24  \\ \hline
1 &  221 & 229 & 54 & 147 & 179  \\ \hline
1 &  150 & 123 & 195 & 243 & 119  \\ \hline
\end{array}
$$

Thus we can state the following.
\begin{prop}\label{>109}
There exists a $({q-1\over2},6;6)$ Heffter space for every prime $q\equiv13$ $($mod $24)$ in the range $[109,8000]$.
\end{prop}

\section{Main result}

Now we show that Theorem \ref{general} is always successfull provided that the prime power
$q$ is sufficiently large.  

\begin{thm}\label{e=k}
There exists a simple $({q-1\over2},k;r)$ Heffter space
for any prime power $q\equiv2k+1$ $($mod $4k)$ greater than $Q(k,r)$ where
$$Q(k,r)={1\over4}\biggl{[}(2k-1)^2+\sqrt{(2k-1)^4+16k(kn+1)}\biggl{]}^2 \quad
\mbox{with $n=(2r-1)(k-2)+r-1$}.$$
\end{thm}
\begin{proof}
By Theorem \ref{general},
it is enough to show that there exists a simple $(V,k,r)$ Heffter difference matrix with $V$
the standard half-cyclotomic set of order $2k$. We prove this by induction on $r$.
The empty matrix can be viewed as a $(V,k,0)$ Heffter difference matrix.
Assume that $q>Q(k,r)$ and that there exists a simple $(V,k,r-1)$ Heffter difference matrix $A$.

We first construct, iteratively,  a $(k-2)$-tuple $(b_0,b_1,\dots,b_{k-3})$ as follows.
Take $b_0$ in $C^{2k}_0$ arbitrarily and then take the other elements
$b_1$, $b_2$, \dots, $b_{k-3}$, one by one, according to the rule that
once that $b_{j-1}$ has been chosen, we pick $b_j$ arbitrarily in the set $C^{2k}_{j}\setminus (Y_j \ \cup \ Y'_j)$
where $$Y_j=\biggl{\{}b_i\cdot{a_{h,j}\over a_{h,i}} \ | \ 1\leq h\leq r-1; 0\leq i\leq j-1\biggl{\}} \quad{\rm and}\quad Y'_j=\biggl{\{}-\sum_{i=h}^{j-1}b_i \ | \ 0\leq i\leq j-1\biggl{\}}.$$
This is certainly possible since $Y_j \ \cup \ Y'_j$
has size at most equal to $r(k-3)$ which, in view of the hypothesis $q>Q(k,r)$, is easily seen much less that than the size 
of $C^{2k}_{j}$ that is ${q-1\over 2k}$.

Now let $s$ be the sum of all elements of the constructed $(k-2)$-tuple and consider the sets 
$$X=\{x\in C^{2k}_{k-2} \ : \ x+s\in C^{2k}_{k-1}\} \quad{\rm and}\quad Z=Z_1 \ \cup \ Z_2 \ \cup \ Z_3 \ \cup \ Z_4$$
where
$$Z_1=\biggl{\{}b_j\cdot {a_{h,k-2}\over a_{h,j}} \ | \ 1\leq h\leq r-1; 0\leq j\leq k-3\biggl{\}};$$
\medskip
$$Z_2=\biggl{\{}-s-b_j\cdot{a_{h,k-1}\over a_{h,j}} \ | \ 1\leq h\leq r-1; 0\leq j\leq k-3\biggl{\}};$$
\medskip
$$Z_3=\biggl{\{}-s\cdot{a_{h,k-2}\over a_{h,k-2}+a_{h,k-1}} \ | \ 1\leq h\leq r-1\biggl{\}};$$
\medskip
$$Z_4=\biggl{\{}-\sum_{j=i}^{k-3}b_j \ | \ 0\leq i\leq k-3\biggl{\}}.$$

As a special case of Theorem 1.1 in \cite{BP}, it is possible to see that the hypothesis $q>Q(k,r)$
implies that $X$ has size greater than $n=(2r-1)(k-2)+r-1$.
Also note that the sets $Z_1$, $Z_2$ have size at most equal to $(r-1)(k-2)$, that $Z_3$ has size at most equal to $r-1$, and that $Z_4$ has size equal to $k-2$. 
Thus we have $|Z|\leq n<|X|$ so that there exists an element $x\in X$ not belonging to $Z$.
Take such an element $x$ and consider the $r\times k$ matrix $B$ obtained from $A$ by adding the row
$(b_0,b_1,\dots,b_{k-3},x,-x-s)$.

This row is zero-sum by the definition of $s$.
The fact that $b_j\notin Y'_j$ and that $x\notin Z_4$ implies that its partial sums are pairwise distinct.

For $0\leq i<j\leq k-3$, we see that $B^j/B^i$ does not have repeated elements because $b_j\notin Y_j$.
For $0\leq j\leq k-3$, we see that $B^{k-2}/B^j$ does not have repeated elements because $x\notin Z_1$.
For $0\leq j\leq k-3$, we see that $B^{k-1}/B^j$ does not have repeated elements because $x\notin Z_2$.
Finally, $B^{k-1}$/$B^{k-2}$ does not have repeated elements because $x\notin Z_3$.
 
We conclude that $B$ is a simple $(V,k,r)$ Heffter difference matrix and the assertion follows.
\end{proof}

Given that the bound $Q(k,r)$ is a little bit unwieldy, we looked for a more manageable approximation by excess 
which is given by $8k^4r$. Thus we can state the following.

\begin{cor}\label{e=k bis}
There exists a simple $({q-1\over2},k;r)$ Heffter space
for any prime power $q\equiv2k+1$ $($mod $4k)$ greater than $8k^4r$.
\end{cor}

As another immediate corollary, we obtain our main result Theorem \ref{main}.

\medskip
\noindent
{\bf Proof of Theorem \ref{main}}.\quad The assertion follows from Theorem \ref{e=k} and the fact that there are 
infinitely many primes $q\equiv2k+1$ (mod $4k$) by Dirichlet's theorem on arithmetic progressions.



\section{Connections with orthogonal cycle systems}

We recall that a $k$-cycle system of order $v$ is a set of $k$-cycles whose edges partition
the edge set of the complete graph $K_v$. Two cycle systems of $K_v$ are {\it orthogonal} if there is no $k$-cycle of one system
sharing more than one edge with the other.

The construction of a set of mutually orthogonal cycle systems was recently considered in \cite{BP2,BCP,KY}.
As explained in \cite{BP2}, every simple $(v,k;r)$ Heffter space determines a set
of $r$ mutually orthogonal $k$-cycle systems of order $2v+1$. Hence, as a
consequence of the results of the previous section, we get the following.

\begin{thm}
If $q=2kn+1$ is a prime power with $k\geq3$ and $n$ odd, then there exist at least $\lfloor {n\over k^3}\rfloor$
mutually orthogonal $k$-cycle systems of order $q$.
\end{thm}
\begin{proof}
The hypothesis that $n$ is odd implies that $q\equiv 2k+1$ (mod $4k$). Also note that
we can write $$q=2kn+1=2k\cdot4k^3\cdot{n\over 4k^3}+1>8k^4\biggl{\lfloor}{n\over k^3}\biggl{\rfloor}.$$
Then, by Corollary \ref{e=k bis}, there exists a simple $({q-1\over2},k;\lfloor{n\over k^3}\rfloor)$ Heffter space.
The assertion follows from what we recalled above.
\end{proof}

For $k$ odd, the above lower bound on the number of mutually orthogonal $k$-cycle systems of order a prime power  $q=2kn+1$ with $n$ odd
is better than the one obtained in \cite{BP2} , that is $\lceil {n\over k^4}\rceil$.
On the other hand, for $k$ even, it is much worse than the bound deducible from the main result in \cite{BCP},
that is $4n$ for $k=4$, $\lceil {n\over 4k-2}\rceil$ for $4\neq k\equiv0$ (mod 4), $\lceil {n\over 24k-18}\rceil$ for $k\equiv2$ (mod 4).

We point out, however, that for small values of $n$ all the above bounds are quite bad apart from the special case $k=4$. 
In these ``small cases" to apply our Theorem \ref{general} with the aid of a computer seems to be promising. 
For instance \cite{BCP} guarantees the existence of six mutually orthogonal 6-cycle systems of order a prime $q\equiv13$ (mod 24)
starting only from $q=7573$. On the other hand our Proposition \ref{>109} guarantees 
this existence starting from $q=109$.

\section*{Acknowledgements}
The authors are partially supported by INdAM - GNSAGA.


\begin{thebibliography}{99}


\bibitem{AL} B. Alspach, G. Liversidge,  {\it On strongly sequenceable abelian groups}, Art Discrete Appl. Math. \textbf{3} (2020), \#P1.02.

\bibitem{A} D.S. Archdeacon,
{\it Heffter arrays and biembedding graphs on surfaces},
Electron. J. Combin. \textbf{22} (2015), \#P1.74.

\bibitem{ADDY} D.S. Archdeacon, J.H. Dinitz, D.M. Donovan, E.\c{S}. Yaz\i c\i,
{\it Square integer Heffter arrays with empty cells},
Des. Codes Cryptogr. \textbf{77} (2015), 409--426.


\bibitem{BJL} T. Beth, D. Jungnickel, H. Lenz, Design Theory. Cambridge University Press, Cambridge, 1999.

\bibitem{BGL}
D. Bryant, H. Gavlas, A. Ling, {\it Skolem-type difference sets for cycles}, 
Electronic J. Combin. {\bf 10} (2003), $\sharp$R38.

\bibitem{B} M. Buratti, Tight Heffter arrays from finite fields,
In {\it Stinson 66 -- New Advances in Designs, Codes and Cryptography}, 
C.J. Colbourn and J.H. Dinitz (eds.), Fields Institute Communications {\bf86} (2024), 25--36.

\bibitem{BurDel}
M. Buratti, A. Del Fra, 
{\it Existence of cyclic k-cycle systems of the complete graph}, Discrete Math. {\bf261} (2003), 113--125.


\bibitem{BJMP} M. Buratti, L. Johnson, L. Mella, A. Pasotti,
{\it On relative Heffter spaces}, in preparation.

\bibitem{BMN} M. Buratti, F. Merola, A. Nakic,
{\it Additive combinatorial designs}, in preparation.

\bibitem{BN1}
M. Buratti, A. Nakic,
{\it Super-regular Steiner $2$-designs}, Finite Fields Appl. {\bf85} (2023), Article number 102116.

\bibitem{BN2}
M. Buratti, A. Nakic,  {\it Additivity of symmetric and subspace $2$-designs}, to appear in Des. Codes Cryptogr.  (available at arXiv:2307.08134)

\bibitem{BP} M. Buratti, A. Pasotti,
{\it Combinatorial designs and the theorem of Weil on multiplicative character sums}, 
Finite Fields Appl.  {\bf15}  (2009),  332--344.

\bibitem{BP2} M. Buratti, A. Pasotti, {\it Heffter spaces}, Finite Fields Appl.  {\bf98} (2024), Article number 102464.

\bibitem{BP3} M. Buratti, A. Pasotti,
{\it Integer Heffter spaces}, in preparation.

\bibitem{BS22} M. Buratti, D.R. Stinson,
{\it On resolvable Golomb rulers, symmetric configurations and
progressive dinner parties},
J. Algebraic Combin. \textbf{55} (2022), 141--156.

\bibitem{BCP} A.C. Burgess, N.J. Cavenagh, D.A. Pike,
{\it Mutually orthogonal cycle systems}, 
Ars Math. Contemp. \textbf{23} (2023), \#P2.05.


\bibitem{CFP}
A. Caggegi, G. Falcone, M. Pavone,
{\it On the additivity of block designs}, J. Algebr. Comb. {\bf 45} (2017), 271--294.

\bibitem{CFP2}
A. Caggegi, G. Falcone, M. Pavone, 
{\it Additivity of affine designs}, 
J.  Algebr. Comb. {\bf53} (2021), 755--770.

\bibitem{CDDY} N.J. Cavenagh, J.H. Dinitz, D. Donovan, E.\c{S}. Yaz\i c\i,
{\it The existence of square non-integer Heffter arrays},
Ars Math. Contemp. \textbf{17} (2019), 369--395.

 \bibitem{CD} C.J. Colbourn, J.H. Dinitz,  Handbook of Combinatorial Designs. Second Edition, Chapman \& Hall/CRC, Boca Raton, FL, 2006.



\bibitem{CMPPSums} S. Costa, F. Morini, A. Pasotti, M.A. Pellegrini,
{\it A problem on partial sums in abelian groups},
Discrete Math. \textbf{341} (2018), 705--712.


\bibitem{DW} J.H. Dinitz, I.M.  Wanless,
{\it The existence of square integer Heffter arrays},
Ars Math. Contemp. \textbf{13} (2017), 81--93.

\bibitem{FP}
G. Falcone, M. Pavone,
{\it Binary Hamming codes and Boolean designs},
Des. Codes Cryptogr. {\bf89} (2021), 1261--1277.

\bibitem{G} G. G\'evay,
{\it Resolvable configurations},
Discr. Appl. Math. \textbf{266} (2019), 319--330.

\bibitem{KY} S. Kukukcifci, E.\c{S}. Yaz\i c\i,  
{\it Orthogonal cycle systems with cycle length less than 10},
J. Combin. Des. \textbf{32} (2024), 31--45. 

\bibitem{N} A. Nakic, private communication. 

\bibitem{PD}  A. Pasotti, J.H. Dinitz,
A survey of Heffter arrays.
In {\it Stinson 66 -- New Advances in Designs, Codes and Cryptography}, 
C.J. Colbourn and J.H. Dinitz (eds.), Fields Institute Communications {\bf86} (2024), 353--392.


\end{thebibliography}
\end{document}